\documentclass[letterpaper, 10 pt, conference]{ieeeconf}
\usepackage{amsfonts}
\usepackage{amsmath}
\usepackage{amssymb}
\usepackage{color}
\usepackage{cite}
\usepackage{epsfig}
\usepackage{graphics}
\usepackage{graphicx}
\usepackage{ifpdf}
\usepackage{subfigure}
\usepackage{times}

\IEEEoverridecommandlockouts                              
\newtheorem{theorem}{Theorem}

\newtheorem{corollary}[theorem]{Corollary}

\newtheorem{definition}[theorem]{Definition}
\newtheorem{example}[theorem]{Example}

\newtheorem{lemma}[theorem]{Lemma}

\newtheorem{remark}[theorem]{Remark}

\begin{document}

\title{{\LARGE \textbf{From Local Measurements to Network Spectral
Properties: Beyond Degree Distributions}}}
\author{Victor~M.~Preciado and~Ali~Jadbabaie\thanks{%
This work is partially supported by the ONR MURI HUNT and AFOR MURI Complex
Networks programs.}\thanks{%
Victor~M.~Preciado and~Ali~Jadbabaie are with GRASP Laboratory, School of
Engineering and Applied Science, University of Pennsylvania, Philadelphia,
PA 19104, USA \texttt{\footnotesize \{preciado,jadbabai\}@seas.upenn.edu}}}
\maketitle

\begin{abstract}
It is well-known that the behavior of many dynamical processes running on
networks is intimately related to the eigenvalue spectrum of the network. In
this paper, we address the problem of inferring global information regarding
the eigenvalue spectrum of a network from a set of local samples of its
structure. In particular, we find explicit relationships between the
so-called spectral moments of a graph and the presence of certain small
subgraphs, also called motifs, in the network. Since the eigenvalues of the
network have a direct influence on the network dynamical behavior, our
result builds a bridge between local network measurements (i.e., the
presence of small subgraphs) and global dynamical behavior (via the spectral
moments). Furthermore, based on our result, we propose a novel decentralized
scheme to compute the spectral moments of a network by aggregating local
measurements of the network topology. Our final objective is to understand
the relationships between the behavior of dynamical processes taking place
in a large-scale complex network and its local topological properties.
\end{abstract}




\section{Introduction}

\label{sec:introduction}

Research in complex networks has important applications in today's massive
networked systems, including the Internet, the World-Wide Web (WWW), as well
as social, biological and chemical networks \cite{S01}. The availability of
massive databases, and reliable tools for data analysis provides a powerful
framework to explore structural properties of large-scale networks \cite%
{BLMCH06}. In many real-world cases, it is impossible to efficiently
retrieve and/or store the exact structure of a complex network due to, for
example, a prohibitively large network size or privacy/security concerns. On
the other hand, it is often possible to gather a great deal of information
by examining local samples of the graph topology, such as the degree
distribution, which are usually easy to collect. In this context, a
challenging problem is to find relationships between the set of local
samples of the network structure and its global functionality.

Many dynamical processes on networks, such as random walks \cite{Ald82},
virus/rumor spreading \cite{DGM08},\cite{PJ09}, and synchronization of
oscillators \cite{PC98},\cite{PV05}, are interesting to study in the context
of large-scale complex networks. The behavior of many of these processes are
intimately related to the eigenvalue spectra of the underlying graph
structure \cite{P08},\cite{PZJP10}. For example, the speed of spreading of a
virus is directly related with the spectral radius of the adjacency matrix
of the network \cite{DGM08}. Hence, spectral graph theory provides us with a
framework to study the relationship between the network structure and its
dynamical behavior.

Network motifs are small subgraphs that are present in a network with a much
higher frequently than in random networks with the same degree sequence \cite%
{MSIKCA02}. There is both empirical and theoretical evidence showing that
these subgraphs play a key role in the network's function and organization 
\cite{Alon07}. One of the main objectives of this paper is to explicitly
relate global properties of a given network with the presence of certain
small subgraphs that can be counted via \emph{local measurements}. We focus
our attention on global properties related with the network's eigenvalues,
in particular, the so-called spectral moments. Since the spectral properties
are known to have a direct influence on the network dynamical behavior, our
result builds a bridge between local network measurements (i.e., the
presence of small subgraphs) and global dynamical behavior (via the spectral
moments). It is also worth pointing out that, although a set of spectral
moments is not enough to completely describe the spectral distribution of a
network, it allows us to extract a great deal of information. For example,
Popescu and Bertsimas provide in \cite{PB05} an optimization framework for
computing optimal bounds on the properties of a distribution from moments
constrains. More generally, there is a variety of techniques that can be
applied to extract spectral information from a truncated sequence of
spectral moments \cite{P08}.

Based on our results, we also propose a novel decentralized algorithm to
efficiently aggregate a set of local network measurements into global
spectral moments. Our work is related to \cite{KM08}, where a fully
distributed algorithm is proposed to compute the full set of eigenvalues and
eigenvectors of a matrix representing the network topology. In contrast, our
approach is computationally much cheaper, since it does not require a
complete eigenvalue decomposition. Furthermore, our approach also provides a
clearer view of the role of certain subgraphs in the network's dynamical
behavior.

The rest of this paper is organized as follows. In the next section, we
review graph-theoretical terminology and introduce definitions needed in our
derivations. In Section~\ref{sec:compute_moments}, we derive explicit
relationships between the moments of the eigenvalue spectrum and local
network measurements. Based on these expressions, we introduce a distributed
algorithm to compute these moments in Section \ref{Distributed Algorithm}.
We conclude the paper mentioning some future work.

\section{Definitions \& Notation}

\label{sec:problem}

Let $\mathcal{G}=\left( \mathcal{V},\mathcal{E}\right) $ denote a simple
undirected graph (with no self-loops) on $n$ nodes, where $\mathcal{V}\left( 
\mathcal{G}\right) =\left\{ v_{1},\dots ,v_{n}\right\} $ denotes the set of
nodes and $\mathcal{E}\left( \mathcal{G}\right) \subseteq \mathcal{V}\left( 
\mathcal{G}\right) \times \mathcal{V}\left( \mathcal{G}\right) $ is the set
of undirected edges. If $\left\{ v_{i},v_{j}\right\} \in \mathcal{E}\left( 
\mathcal{G}\right) $ we call nodes $v_{i}$ and $v_{j}$ \emph{adjacent} (or
neighbors), which we denote by $v_{i}\sim v_{j}$. The set of all nodes
adjacent to a node $v\in \mathcal{V}\left( \mathcal{G}\right) $ constitutes
the \emph{neighborhood} of node $v$, defined by $\mathcal{N}^{v}=\{w\in 
\mathcal{V}\left( \mathcal{G}\right) :\left\{ v,w\right\} \in \mathcal{E}%
\left( \mathcal{G}\right) \}$, and the number of those neighbors is called
the \emph{degree} of node $v$, denoted by $\deg v$ or $d_{v}$.

We define a \emph{walk} of length $k$ from $v_{0}$ to $v_{k}$ to be an
ordered sequence of nodes $\left( v_{0},v_{1},...,v_{k}\right) $ such that $%
v_{i}\sim v_{i+1}$ for $i=0,1,...,k-1$. If $v_{0}=v_{k}$, then the walk is
closed. A closed walk with no repeated nodes (with the exception of the
first and last nodes) is called a \emph{cycle}. \emph{Triangles} and \emph{%
quadrangles} are cycles of length three and four, respectively. We say that
a graph $\mathcal{G}$ is \emph{connected} if there exists a walk between
every pair of nodes. Let $d\left( v,w\right) $ denote the \emph{distance}
between two nodes $v$ and $w$, i.e., the minimum length of a walk from $v$
to $w$. We define the diameter of a graph, denoted by $diam\left( \mathcal{G}%
\right) $, as the maximum distance between any pair of nodes in $\mathcal{G}$%
. We say that $v$ and $w$ are $k$-th order neighbors if $d\left( v,w\right)
=k,$ and define the $k$-th order neighborhood of a node $v$ as the set of
nodes within a distance $k$ from $v$, i.e., $\mathcal{N}_{k}^{v}=\left\{
w\in \mathcal{V}\left( \mathcal{G}\right) :d\left( v,w\right) \leq k\right\} 
$. A $k$-th order neighborhood, induces a subgraph $\mathcal{G}%
_{k}^{v}\subseteq \mathcal{G}$ with node-set $\mathcal{N}_{k}^{v}$ and
edge-set $\mathcal{E}_{k}^{v}$ defined as the subset of edges of $\mathcal{E}%
\left( \mathcal{G}\right) $ that connect two nodes in $\mathcal{N}_{k}^{v}$.

Graphs can be algebraically represented via the \emph{adjacency} matrix. The 
\emph{adjacency matrix} of an undirected graph $\mathcal{G}$, denoted by $A_{%
\mathcal{G}}=[a_{ij}]$, is an $n\times n$ symmetric matrix defined
entry-wise as $a_{ij}=1$ if nodes $v_{i}$ and $v_{j}$ are adjacent, and $%
a_{ij}=0$ otherwise\footnote{%
For simple graphs with no self-loops, $a_{ii}=0$ for all $i$.}.


\section{Spectral Analysis via Subgraph Embedding}

\label{sec:compute_moments}

\bigskip

In this section, we derive an explicit relationship between the spectral
moments of the adjacency matrix and the presence of certain subgraphs in $%
\mathcal{G}$. We say that a graph $\mathcal{H}$ is embedded in $\mathcal{G}$
if $\mathcal{H}$ is isomorphic\footnote{%
An isomorphism of graphs $\mathcal{G}$ and $\mathcal{H}$ is a bijection
between the vertex sets $\mathcal{V}(\mathcal{G})$ and $\mathcal{V}(\mathcal{%
H})$, $f:\mathcal{V}(\mathcal{G})\rightarrow \mathcal{V}(\mathcal{H})$, such
that any two vertices $u$ and $v$ of $\mathcal{G}$ are adjacenct in $%
\mathcal{G}$ if and only if $f(u)$ and $f(v)$ are adjacent in $\mathcal{H}$.}
to a subgraph in $\mathcal{G}$. The embedding frequency of $\mathcal{H}$ in $%
\mathcal{G}$, denoted by $F(\mathcal{H},\mathcal{G})$, is the number of
different subgraphs of $\mathcal{G}$ to which $\mathcal{H}$ is isomorphic.
The term network motif is used to designate those subgraphs of $\mathcal{G}$
that occur with embedding frequencies far higher than in random networks
with the same degree sequence \cite{MSIKCA02}. Theoretical and experimental
evidence shows that some of these motifs carry significant information about
the network's function and organization \cite{Alon07}. In this section, we
derive an explicit expression for the spectral moments as a linear
combination of the embedding frequencies of certain subgraphs. Our results
provide a direct relationship between the presence of network motifs and
global properties of the network, in particular, spectral moments. In the
coming subsections, we first provide a theoretical foundation for our
analysis. Second, we derive explicit expressions for the spectral moments in
terms of network metrics, such as the degree sequence or the number of
triangles in the graph. In the next Section, we propose a decentralized
algorithm to compute the spectral moments of a network based on
decentralized subgraph counting.

\subsection{Spectral Moments and Subgraph Embedding}

Algebraic graph theory provide us with the tools to relate topological
properties of a graph with its spectral properties. In particular, we are
interested in studying the spectral moments of the adjacency matrix of a
given graph. We denote by $\{\lambda _{i}\}_{i=1}^{n}$ the set of
eigenvalues of $A_{\mathcal{G}}$ and define the $k$-th spectral moment of
the adjacency matrix as%
\begin{equation}
m_{k}\left( A_{\mathcal{G}}\right) =\frac{1}{n}\sum_{i=1}^{n}\lambda
_{i}^{k}.  \label{Spectral Moment}
\end{equation}%
From algebraic graph theory, we have the following result relating the $k$%
-th spectral moment of a graph $\mathcal{G}$ with the number of closed walks
of length $k$ in $\mathcal{G}$ \cite{Big93}:

\begin{lemma}
Let $\mathcal{G}$ be a simple graph. The $k$-th spectral moment of the
adjacency matrix of $\mathcal{G}$ can be written as%
\begin{equation}
m_{k}(A_{\mathcal{G}})=\frac{1}{n}\left\vert \Psi _{\mathcal{G}}^{\left(
k\right) }\right\vert ,  \label{Moments as Walks in Graph}
\end{equation}%
where $\Psi _{\mathcal{G}}^{\left( s\right) }$ is the set of all possible
closed walks of length $k$ in $\mathcal{G}\footnote{%
We denote by $\left\vert Z\right\vert $ the cardinality of a set $Z$.}$.
\end{lemma}

\bigskip

\begin{corollary}
\label{One corollary}Let $\mathcal{G}$ be a simple graph. Denote by $E_{%
\mathcal{G}}$ and $\Delta _{\mathcal{G}}$ the number of edges and triangles
in $\mathcal{G}$, respectively. Then, 
\begin{equation}
m_{1}(A_{\mathcal{G}})=0,\text{ }m_{2}(A_{\mathcal{G}})=2E_{\mathcal{G}}/n,%
\text{ and }m_{3}(A_{\mathcal{G}})=6\,\Delta _{\mathcal{G}}/n.
\label{algebraic graph}
\end{equation}
\end{corollary}

\bigskip

In the following, we develop on the above Lemma to derive an expression for $%
m_{k}\left( A_{\mathcal{G}}\right) ,$ for any $k$, as a linear combination
of the embedding frequencies of certain subgraphs. Although the first 3
moments present very simple expressions in Corollary \ref{One corollary},
the larger the $k$, the more involved those expressions become. In what
follows, we describe a systematic procedure to derive these expressions
efficiently.

First, we need to introduce some nomenclature. Given a walk of length $k$ in 
$\mathcal{G}$, $w=\left( v_{0},v_{1},...,v_{k}\right) $, we denote its
node-set as $\mathcal{V}\left( w\right) =\left\{ v_{0},...,v_{k}\right\} $
and its edge-set as $\mathcal{E}\left( w\right) =\cup _{i=1}^{k}\left\{
v_{i-1},v_{i}\right\} $. Hence, we define the underlying simple graph of a
walk $w$ as $\mathcal{H}\left( w\right) =\left( \mathcal{V}\left( w\right) ,%
\mathcal{E}\left( w\right) \right) $. We say that a simple subgraph $%
h\subseteq \mathcal{G}$ spans the walk $w$ if $h$ is isomorphic to $\mathcal{%
H}\left( w\right) $. Notice how different walks can share the same
underlying simple graph. We also denote by $\mathcal{I}\left( h\right) $ the
unlabeled simple graph isomorphic to a given graph $h$. Applying the
function $\mathcal{I}$ to the underlying graph of a given walk $w$, we
obtain an unlabeled graph $g=\mathcal{I}\left( \mathcal{H}\left( w\right)
\right) $. Notice how different walks, not even sharing the same edge-set,
can share the same \emph{unlabeled} underlying simple graph. For example,
any closed triangular walk $w_{3}=\left( v_{i},v_{j},v_{k},v_{i}\right) $
with $v_{i}\sim v_{j}\sim v_{k}\sim v_{i}$ gives rise to the same unlabeled
graph $\mathcal{I}\left( \mathcal{H}\left( w_{3}\right) \right) $, namely,
an unlabeled triangle.

We denote the set of closed walks of length $k$ in $\mathcal{G}$ sharing the
same underlying simple graph $h$ as $\mathcal{R}_{h}=\left\{ w\in \Psi _{%
\mathcal{G}}^{\left( k\right) }\text{ s.t. }\mathcal{H}\left( w\right)
=h\right\} $. Inversely, consider the set of all closed walks of length $k$
in $\mathcal{G}$, $\Psi _{\mathcal{G}}^{\left( k\right) }$. We can define
another set, which we denote by $\mathbf{H}_{k}$, containing all possible
simple graphs spanning walks in $\Psi _{\mathcal{G}}^{\left( k\right) }$,
i.e., $\mathbf{H}_{k}=\left\{ \mathcal{H}\left( w_{k}\right) ,w_{k}\in \Psi
_{\mathcal{G}}^{\left( k\right) }\right\} $. Furthermore, we define the set
of \emph{unlabeled} graph associated with walks of length $k$ as $\mathbf{I}%
_{k}=\left\{ \mathcal{I}\left( \mathcal{H}\left( w_{k}\right) \right)
,w_{k}\in \Psi _{\mathcal{G}}^{\left( k\right) }\right\} $. In what follows,
we derive an expression for the $k$-th spectral moment as a linear
combination of the embedding frequencies of the (unlabeled) graphs in $%
\mathbf{I}_{k}$.

Note that the mapping $\mathcal{H}:\Psi _{\mathcal{G}}^{\left( k\right)
}\rightarrow \mathbf{H}_{k}$, is a surjection onto $\mathbf{H}_{k}$ that
induces a partition in $\Psi _{\mathcal{G}}^{\left( k\right) }$, namely,
each graph $h\in \mathbf{H}_{k}$ (with labeled nodes), induces a partition
subset $\mathcal{R}_{h}\subseteq \Psi _{\mathcal{G}}^{\left( k\right) }$.
Similarly, the mapping $\mathcal{I}:\mathbf{H}_{k}\rightarrow \mathbf{I}_{k}$
is a surjection onto $\mathbf{I}_{k}$ and induces a partition in $\mathbf{H}%
_{k}$. Each unlabeled graph $g\in \mathbf{I}_{k}$ defines a partition subset 
$\mathcal{S}_{g}\subseteq \mathbf{H}_{k}$ defined as $\mathcal{S}%
_{g}=\left\{ h\in \mathbf{H}_{k}\text{ s.t. }\mathcal{I}\left( h\right)
=g\right\} $. Note that for $g\in \mathbf{I}_{k}$, $\left\vert \mathcal{S}%
_{g}\right\vert $ is the number of subgraphs of $\mathcal{G}$ isomorphic to $%
g$; hence, $\left\vert \mathcal{S}_{g}\right\vert $ equals the embedding
frequency $F\left( g,\mathcal{G}\right) $. This observation shall be
important in the next Section, where we design a decentralized algorithm to
compute spectral moments.

Based on the above, we derive the following result regarding the spectral
moments of $\mathcal{G}$:

\begin{theorem}
\label{Motif Theorem}Given a simple graph $\mathcal{G}$, its $k$-th spectral
moment can be written as%
\begin{equation}
m_{k}(A_{\mathcal{G}})=\frac{1}{n}\sum_{g\in \mathbf{I}_{k}}\omega
_{g}^{\left( k\right) }~F\left( g,\mathcal{G}\right) .
\label{Moments from Motifs}
\end{equation}%
where $\omega _{g}^{\left( k\right) }=\left\vert \left\{ w\in \Psi
_{g}^{\left( k\right) }:\mathcal{I}\left( \mathcal{H}\left( w\right) \right)
=g\right\} \right\vert $; in other words, $\omega _{g}^{\left( k\right) }$
is the number of all possible closed walks of length $k$ in $g$ with
underlying unlabeled graphs isomorphic to $g$.
\end{theorem}

\begin{proof}
Consider the composition function $\mathcal{J}=\mathcal{I}\circ \mathcal{H}$%
. Since $\mathcal{J}:\Psi _{\mathcal{G}}^{\left( k\right) }\rightarrow 
\mathbf{I}_{k}$ is a composition of surjective functions, it is itself a
surjective function. Thus, $\mathcal{J}$ induces a trivial partition in $%
\Psi _{\mathcal{G}}^{\left( k\right) }$ as follows. Given an unlabeled graph 
$g\in \mathbf{I}_{k}$, we define a partition subset $\mathcal{T}%
_{g}\subseteq \Psi _{\mathcal{G}}^{\left( k\right) }$ as $\mathcal{T}%
_{g}=\left\{ w\in \Psi _{\mathcal{G}}^{\left( k\right) }\text{ s.t. }%
\mathcal{J}\left( w\right) =g\right\} =\bigcup\nolimits_{h\in \mathcal{S}%
_{g}}R_{h}$. Based on this partition, we can compute the number of closed
walks of length $k$ in $\mathcal{G}$, $\left\vert \Psi _{\mathcal{G}%
}^{\left( k\right) }\right\vert $, as%
\begin{equation*}
\left\vert \Psi _{\mathcal{G}}^{\left( k\right) }\right\vert =\sum_{g\in 
\mathbf{I}_{k}}\left\vert \mathcal{T}_{g}\right\vert =\sum_{g\in \mathbf{I}%
_{k}}\sum_{h\in \mathcal{S}_{g}}\left\vert R_{h}\right\vert .
\end{equation*}%
Note that, for a particular (labeled) graph $h\in \mathbf{H}_{k}$, the
number of closed walks of length $k$ such that their underlying simple graph
is $h$, i.e., $\left\vert R_{h}\right\vert $,\ is a quantity that depends
exclusively on the topology of the unlabeled version of $h$, i.e., $g=%
\mathcal{I}\left( h\right) $. We define this quantity as $\omega _{k}\left(
g\right) =\left\vert R_{h}\right\vert $. Hence, we can rewrite the above as%
\begin{eqnarray*}
\left\vert \Psi _{\mathcal{G}}^{\left( k\right) }\right\vert &=&\sum_{g\in 
\mathbf{I}_{k}}\left\vert R_{h}\right\vert \sum_{h\in \mathcal{S}_{g}}1 \\
&=&\sum_{g\in \mathbf{I}_{k}}\omega _{g}^{\left( k\right) }\left\vert
S_{g}\right\vert ,
\end{eqnarray*}%
where we have used the fact that $\left\vert R_{h}\right\vert =\omega
_{g}^{\left( k\right) }$ in the last equality. Therefore, substituting the
above in (\ref{Moments as Walks in Graph}) we obtain the expression for the
spectral moments in the statement of the theorem.
\end{proof}

\begin{figure}[tbp]
\centering
\includegraphics[width=1.0\linewidth]{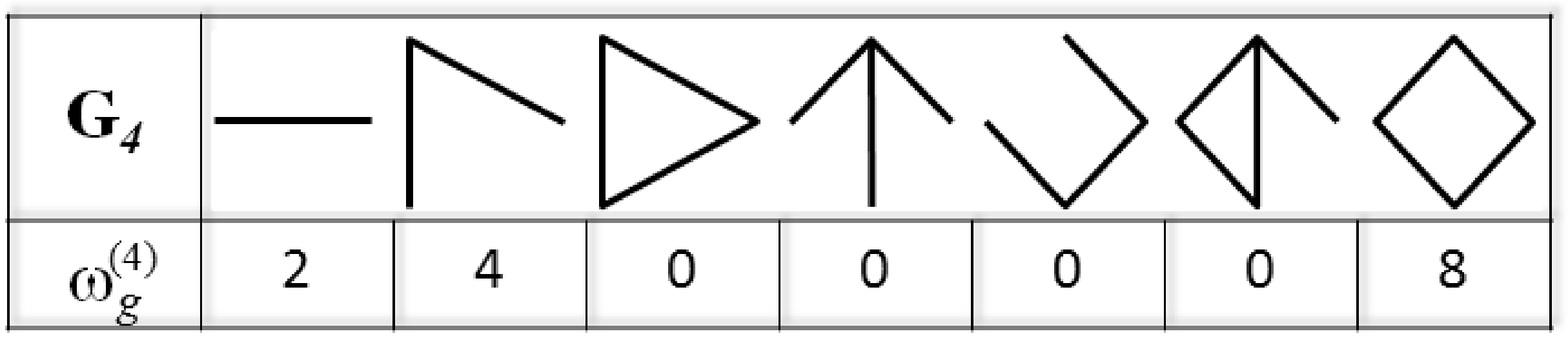} 
\caption{In this table, we represent the set of nonisomorphic connected
graph with at most 4 nodes and 4 edges. The set of subgraphs in $\mathbf{I}%
_{4}$ are those with corresponding coefficient $\protect\omega _{g}^{(4)}>0$%
. These subgraphs are involved in the computation of the $4$-th spectral
moment.}
\end{figure}

The expression in the right-hand side of (\ref{Moments from Motifs}) is a
linear combination of the embedding frequencies of the set of subgraphs $%
g\in \mathbf{I}_{k}$. The coefficients $\omega _{k}\left( g\right) $ in this
linear combination are defined as the number of closed walks of length $k$
with unlabeled underlying graph $g$. In what follows, we describe an
algorithm to determine the set $\mathbf{I}_{k}$ and compute $\omega
_{g}^{\left( k\right) }$. In order to compute the first $K$ spectral
moments, we have to study the sets of closed walks $\Psi _{\mathcal{G}%
}^{\left( k\right) }$, for $k\leq K$. For each particular value of $k$, the
set $\mathbf{I}_{k}$ are unlabeled connected graphs with at most $k$ edges.
Also, the maximum number of nodes visited by walks in $\Psi _{\mathcal{G}%
}^{\left( k\right) }$ are equal to $k$. We denote by $\mathbf{G}_{k}$ the
set of all (unlabeled) nonisomorphic connected graphs with at most $k$ nodes
and at most $k$ edges\footnote{%
This set can be easily determined using the command \texttt{ListGraphs}
included in the package \texttt{Combinatorica} included in \texttt{%
Mathematica}.}; hence, $\mathbf{I}_{k}\subseteq \mathbf{G}_{k}$. For
example, we illustrate all the graphs in $\mathbf{G}_{4}$ (excepting the
isolated-node graph) in Fig. 1.

In order to compute the $k$-th moment via (\ref{Moments from Motifs}), we
need to compute $\omega _{g}^{\left( k\right) }$ for all $g\in \mathbf{I}%
_{k} $, where $\omega _{g}^{\left( k\right) }$ can be interpreted as the the
number of closed walks of length $k$ in $g$ that visit all the nodes and
edges of $g$ (hence, its underlying simple graph is $g$). The computational
complexity of computing $\omega _{g}^{\left( k\right) }$ is the same as the
one of counting the number of Eulerian walks in an undirected (multi)graph.
This counting problem is known to be $\#P$-complete \cite{BW04}. Hence,
there is not a closed-form expression for $\omega _{g}^{\left( k\right) }$,
and it has to be computed via a brute-force combinatorial exploration over
all the possible closed walks of length $k$ visiting all the nodes and edges
of the subgraph $g$. This exploration can be performed in reasonable time
for subgraphs $g$ of small and medium size -- which are the ones we are
interested in. Although computationally expensive, this computation is done
once and for all for a particular $g$. In other words, once $\omega
_{g}^{\left( k\right) }$ is computed, the coefficient for the spectral
moment in (\ref{Moments from Motifs})\ are known for any given $\mathcal{G}$.

\begin{figure*}[tbp]
\centering
\includegraphics[width=0.9\linewidth]{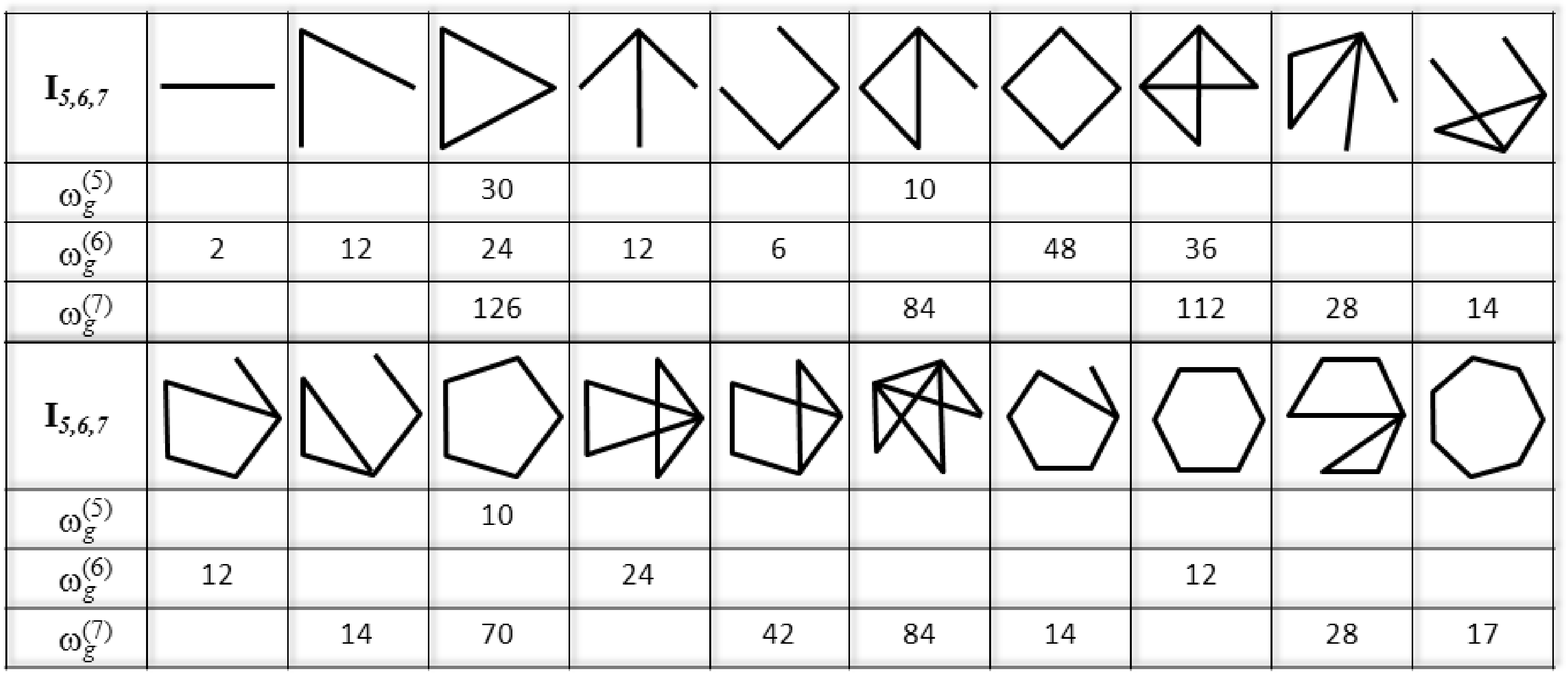} 
\caption{In the above table, we represent the set of nonisomorphic connected
graphs with at most 7 nodes and 7 edges. The set of subgraphs in $\mathbf{I}%
_{5}$, $\mathbf{I}_{6}$, and $\mathbf{I}_{7}$ are those with corresponding
coefficients $\protect\omega _{g}^{(5)},\protect\omega _{g}^{(6)},$ and $%
\protect\omega _{g}^{(7)}>0$, respectively.}
\end{figure*}

For convenience, we provide the coefficients $\omega _{g}^{\left( k\right) }$
for all the graphs in $\mathbf{G}_{4}$ in Fig. 1. Using this table as an
example, we also observe that those graphs $t\in \mathbf{G}_{k}$ such that $%
t\not\in \mathbf{I}_{k}$, have $\omega _{t}^{\left( k\right) }=0$ (since $%
\omega _{t}^{\left( k\right) }$ being zero indicates that there is no walk
in $t$ satisfying the conditions to be part of the set $\mathbf{I}_{k}$).
For convenience, we also provide a list of the graphs in $\cup _{k=5,6,7}%
\mathbf{I}_{k}$ and the associated coefficients $\omega _{g}^{\left(
k\right) }$ in Fig. 2 (where we have left blank those cells where $\omega
_{g}^{\left( k\right) }=0$). In the following paragraphs, we illustrate the
usage of Theorem \ref{Motif Theorem} to derive explicit expressions for the
first $5$ spectral moments of a given graph $\mathcal{G}$.

From the subgraphs and coefficients in Fig. 1, we can compute the $4^{th}$
spectral moment via (\ref{Moments from Motifs}) as follows%
\begin{equation}
m_{4}(A_{\mathcal{G}})=\frac{1}{n}\left[ 2E_{\mathcal{G}}+4\Lambda _{%
\mathcal{G}}+8\Phi _{\mathcal{G}}^{\left( 4\right) }\right] ,
\label{4th Moment}
\end{equation}%
where $E_{\mathcal{G}}$ is the number of edges, $\Lambda _{\mathcal{G}}$ is
the number of wedge-graphs\footnote{%
A wedge graph is isomorphic to a chain graph of length 2.} in $\mathcal{G}$,
and $\Phi _{\mathcal{G}}^{\left( 4\right) }$ is the number of 4-cycles in $%
\mathcal{G}$ (see subgraphs and coefficients in Fig. 1). Thus, in order to
compute the $4^{th}$ spectral moment, we must be able to count the number of
wedge-graphs and 4-cycles in $\mathcal{G}$. Although the number of edges and
cycles in a graph are common network metrics, the number of wedges is not.
It is then convenient to rewrite the number of wedges in terms of more
familiar network metrics. In fact, we can rewrite the number of wedges in
terms of the degree sequence of the graph as follows:%
\begin{equation*}
\Lambda _{\mathcal{G}}=\sum_{j=1}^{n}\binom{\deg j}{2}=\frac{1}{2}%
(W_{2}-W_{1}),
\end{equation*}%
where $W_{r}=\sum_{v=1}^{n}\left( \deg v\right) ^{r}$ is the $r$-th
power-sums of the degree sequence $\left\{ \deg v\right\} _{v=1}^{n}$. Since 
$E_{\mathcal{G}}=\frac{1}{2}\sum_{v=1}^{n}\deg v$ in a simple graph, we can
write \ref{4th Moment} in terms of power-sums and $4$-cycles as:%
\begin{equation}
m_{4}(A_{\mathcal{G}})=\frac{1}{n}\left[ 2W_{2}-W_{1}+8\Phi _{\mathcal{G}%
}^{\left( 4\right) }\right] .  \label{4th Moment Metrics}
\end{equation}%
We illustrate this result in the following example.

\begin{example}
\label{Ring graph}Consider the $n$-ring graph $R_{n}$ (without self-loops).
We know that the eigenvalues of the adjacency matrix of the ring graph $%
A_{R_{n}}$ are $\left\{ 2\cos i\frac{2\pi }{n}\right\} _{i=0}^{n-1}$. Hence,
the 4-th moment is given by the summation $m_{4}(A_{R_{n}})=\frac{1}{n}%
\sum_{i=0}^{n-2}\left( 2\cos i\frac{2\pi }{n}\right) ^{4}$, which (after
some computations) can be found to be equal to $6$ for $n\not\in \left\{
2,4\right\} $. We can apply (\ref{4th Moment}) to easily reach the same
result \emph{without performing an eigenvalue decomposition}, as follows. In
the ring graph, we have $E_{\mathcal{G}}=n,$ $\Lambda _{R_{n}}=n$ for $n\neq
2$, $\Lambda _{R_{2}}=0,$ $\Phi _{R_{n}}^{\left( 4\right) }=0$ for $n\neq 4$%
, and $\Phi _{R_{4}}^{\left( 4\right) }=1$. Hence, we obtain the same value
for $m_{4}(A_{R_{n}})$ directly from (\ref{4th Moment}).
\end{example}

\bigskip

From the subgraphs in Fig. 2, and the row of coefficients $\omega _{g}^{(5)}$%
, we have the following expression for the $5^{th}$ spectral moment via (\ref%
{Moments from Motifs}):%
\begin{equation}
m_{5}(A_{\mathcal{G}})=\frac{1}{n}\left[ 30\Delta _{\mathcal{G}}+10\Upsilon
_{\mathcal{G}}+10\Phi _{\mathcal{G}}^{\left( 5\right) }\right] ,
\label{5th Moment}
\end{equation}%
where $\Delta _{\mathcal{G}}$ and $\Phi _{\mathcal{G}}^{(5)}$ is the number
of triangles and $5$-cycles in $\mathcal{G}$. Also, we represent by $%
\Upsilon _{\mathcal{G}}$ the number of subgraphs isomorphic to the sixth
subgraph in the top row of Fig. 2, counting from the left. Analyzing the
structure of this subgraph, we have that $\Upsilon _{\mathcal{G}%
}=\sum_{i=1}^{n}(d_{i}-2)\Delta _{\mathcal{G}}^{(i)}$, where $\Delta _{%
\mathcal{G}}^{(i)}$ is the number of triangles in $\mathcal{G}$ touching
node $i$. Hence, defining the clustering-degree correlation coefficient as $%
\mathcal{C}_{\Delta d}=\sum_{i=1}^{n}d_{i}\Delta _{\mathcal{G}}^{(i)}$, we
can write $\Upsilon _{\mathcal{G}}=\mathcal{C}_{\Delta d}-6\Delta _{\mathcal{%
G}}$, and the $5^{th}$ spectral moment as:%
\begin{equation}
m_{5}(A_{\mathcal{G}})=\frac{1}{n}\left[ 10\Phi _{\mathcal{G}}^{\left(
5\right) }-30\Delta _{\mathcal{G}}+10\mathcal{C}_{\Delta d}\right] .
\label{5th Moment Metrics}
\end{equation}%
Notice that, apart from triangles and 5-cycles, the clustering-degree
correlation influences the 5-th moment; hence, nontrivial variations of the
local clustering with the degree, as reported in \cite{RSMOB02}, are
relevant in the computation of the 5-th spectral moments.

\begin{remark}
The main advantage of (\ref{Moments from Motifs}) in Theorem \ref{Motif
Theorem} may not be apparent in graphs with simple, regular structure. For these graphs, an
explicit eigenvalue decomposition is available and there may be no need to
look for alternative ways to compute spectral moments. On the other hand, in
the case of large-scale complex networks, the structure of the network is
usually very intricate, in many cases not even known exactly, and an
explicit eigenvalue decomposition can be very challenging, if not
impossible. It is in these cases where the alternative approach proposed in
this paper is most useful.
\end{remark}

As we show in the next section, the spectral moments can be efficiently
computed via a distributed approach from aggregation of local samples of the
graph topology, \emph{without knowing the complete structure of the network}.

\bigskip

\section{Distributed Computation of Spectral Moments}

\label{Distributed Algorithm}

According to (\ref{Moments from Motifs}), the $k$-th spectral moment is
equal to a linear combination of the embedding frequencies $F\left( g,%
\mathcal{G}\right) $ for $g\in \mathbf{I}_{k}$. On the other hand, computing
the embedding frequencies in large-scale networks can be challenging in a
centralized framework, since the computational cost of counting subgraphs
rapidly grows with the network size. In this section, we introduce an
efficient decentralized approach to compute the embedding frequencies of
subgraphs from local samples of the network topology.

During this section, we assume that there is an agent in each node $v\in 
\mathcal{V}\left( \mathcal{G}\right) $ that is able to access its $r$-th
neighborhood, $\mathcal{G}_{r}^{v}$. A naive approach to compute the
embedding frequency of a particular subgraph $g$ would be to compute the
embedding frequency of $g$ in each neighborhood,\ $F(g,\mathcal{G}_{r}^{v})$%
, and sum them up. This approach obviously does not work because this
particular subgraph $g\subseteq \mathcal{G}$ can be in the intersection of
multiple neighborhoods; hence, that subgraph would be counted multiple
times. In what follows, we propose a decentralized counting procedure that
allow us to know how many times a particular subgraph is counted.

First, we need to introduce several definitions:

\begin{definition}
We say that a graphical property $P_{\mathcal{G}}$ is \emph{locally
measurable within a radius }$r$\emph{\ around a node} $v$ if $P_{G}$ is a
function of $\mathcal{N}_{r}^{v}$, i.e., $P_{\mathcal{G}}=f\left( \mathcal{N}%
_{r}^{v}\right) $.
\end{definition}

\begin{definition}
We say that a subgraph $h\subseteq \mathcal{G}$ is \emph{locally countable
within a radius} $r$ \emph{around a node} $v$ if $h\subseteq \mathcal{G}%
_{r}^{v}$.
\end{definition}

For example, both edges and triangles touching a node $v$ are locally
countable within a radius $1$. Also, the number of quadrangles touching $v$
is locally countable within a radius $2$. Furthermore, a wedge is locally
countable within a radius 1 only by the node at the center of the wedge, but
it is not countable by the nodes at the extremes of the wedge. On the other
hand, the wedge is locally countable within a radius 2 by all the nodes in
the wedge. In this examples, we observe how not all the nodes being part of
a subgraph $h\subseteq \mathcal{G}$ have to be able to locally count the
subgraph. In particular, if the radius $r$ is smaller than $\left\lceil
diam\left( h\right) /2\right\rceil $, none of the nodes in $h$ is able to
count the subgraph locally. On the contrary, if the radius $r$ is greater or
equal to the diameter of the subgraph, all the nodes in $h$ are able to
locally count it. In the middle range, $\left\lceil diam\left( h\right)
/2\right\rceil \leq r\leq diam\left( h\right) $, some nodes are able to
locally count $h$ and some are not.

\begin{definition}
For a given value of $r$, the set of \emph{detector nodes} of a given
subgraph $h\subseteq \mathcal{G}$ is defined as%
\begin{equation*}
D_{h}^{\left( r\right) }=\left\{ v\in \mathcal{V}\left( h\right) \text{ s.t. 
}h\subseteq \mathcal{G}_{r}^{v}\right\} .
\end{equation*}
\end{definition}

Note that, although there can be other nodes $u\not\in \mathcal{V}\left(
h\right) $ able to locally count $h$, we do not include them in the set of
detector nodes $D_{h}$. Also note that, given an unlabeled graph $g$, the
size of $D_{h_{i}}^{\left( r\right) }$ for all the subgraphs $h_{i}\subseteq 
\mathcal{G}$ isomorphic to $g$ depends exclusively on the structure of $g$.
In other words, we have that $\left\vert D_{h_{i}}^{\left( r\right)
}\right\vert =\left\vert D_{g}^{\left( r\right) }\right\vert $ for all $%
i=1,...,F\left( g,\mathcal{G}\right) $. For a given $g$, the value of $%
\left\vert D_{g}^{\left( r\right) }\right\vert $ can be algorithmically
computed as follows:

\begin{enumerate}
\item \textbf{Initialize} $M:=0;$

\item \textbf{For} each node $u\in \mathcal{V}\left( g\right) $;

\item \qquad Compute $\delta _{g}(u):=\max_{v\in g}d(u,v)$;

\item \qquad \textbf{If} $\delta _{g}(u)\leq r,$ \textbf{Then} $M:=M+1;$

\item \textbf{End For},

\item $\left\vert D_{g}^{\left( r\right) }\right\vert :=M$.
\end{enumerate}

For convenience, we include the values of $\left\vert D_{g}^{\left( r\right)
}\right\vert $ for $g\in \cup _{k=4,5,6}\mathbf{I}_{k}$ for radius $r=1,2,$
and $3$ in Fig. 4.

\begin{figure*}[tbp]
\centering
\includegraphics[width=0.9\linewidth]{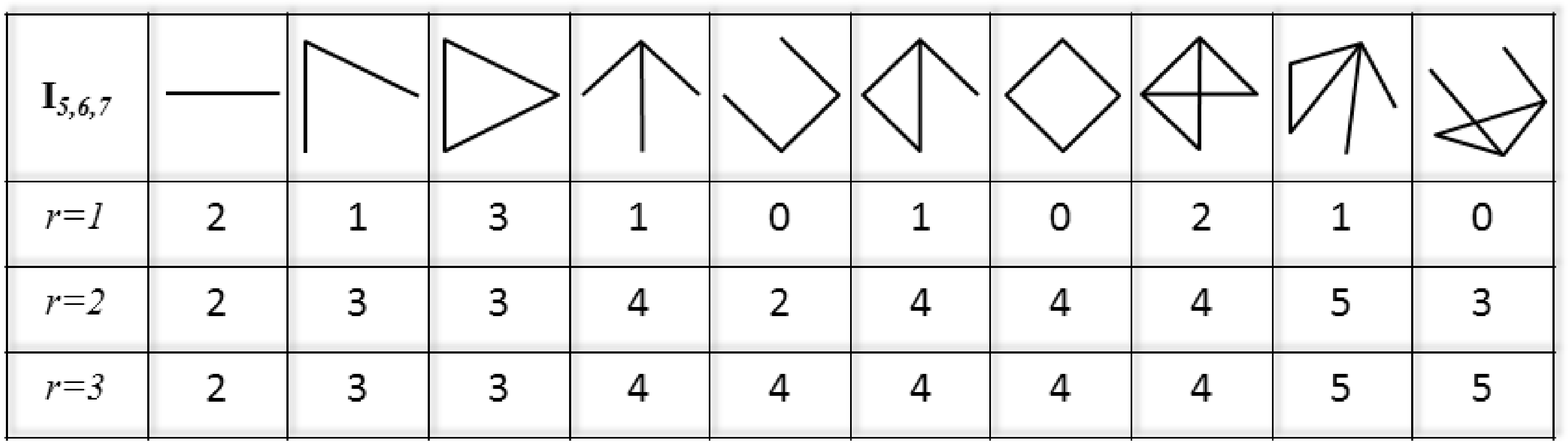} 
\caption{In the above table, we represent several nonisomorphic connected
graphs, and the corresponsing values of $\left\vert D_{g}^{\left( r\right)
}\right\vert $ for $r\in \left\{ 1,2,2\right\} $.}
\end{figure*}

After these preliminary results, we now describe a novel algorithm to
distributedly compute $F\left( g,\mathcal{G}\right) $. Let us consider a
particular node $v\in \mathcal{V}\left( \mathcal{G}\right) $. We define the 
\emph{local embedding frequency} of subgraph $g$ within a neighborhood of
radius $r$ around node $v$, denoted by $H\left( g,\mathcal{N}_{r}^{v}\right) 
$, as the number of different subgraphs $l_{i}\subseteq \mathcal{N}_{r}^{v}$
isomorphic to $g$ such that $v\in \mathcal{V}\left( l_{i}\right) $. Our
distributed algorithm is based on the following result:

\begin{theorem}
\label{Distributed Moment Computation}Let $\mathcal{G}$ be a simple graph.
The spectral moments of the adjacency matrix of $\mathcal{G}$ can be written
as 
\begin{equation}
m_{k}(A_{\mathcal{G}})=\frac{1}{n}\sum_{v\in \mathcal{V}\left( \mathcal{G}%
\right) }\sum_{g\in \mathbf{I}_{k}}\frac{\omega _{k}\left( g\right) }{%
\left\vert D_{g}^{\left( r\right) }\right\vert }H\left( g,\mathcal{N}%
_{r}^{v}\right) .  \label{Distributed Spectral Moment}
\end{equation}
\end{theorem}

\begin{proof}
Note that the local embedding frequency only count subgraphs $l_{i}$
touching node $v$; hence, $H\left( g,\mathcal{N}_{r}^{v}\right) \leq F\left(
g,\mathcal{N}_{r}^{v}\right) $. Furthermore, we have that%
\begin{equation}
\sum_{v\in \mathcal{V}\left( \mathcal{G}\right) }H\left( g,\mathcal{N}%
_{r}^{v}\right) =\left\vert D_{g}\right\vert ~F\left( g,\mathcal{G}\right) ,
\label{local embedding frequencies}
\end{equation}%
since every graph $h\subseteq \mathcal{G}$ isomorphic to $g$ is counted $%
\left\vert D_{g}^{\left( r\right) }\right\vert $ times in the above
summation. Therefore,\ substituting (\ref{local embedding frequencies}) in (%
\ref{Moments from Motifs}), we obtain the statement of the Theorem.
\end{proof}

\bigskip

Based on Theorem \ref{Distributed Moment Computation} it is straightforward
to compute the $k$-th spectral moment distributedly as follows. First,
define the following local variables%
\begin{equation*}
\mu _{k}^{(r)}\left( v\right) \triangleq \sum_{g\in \mathbf{I}_{k}}\frac{%
\omega _{k}\left( g\right) }{\left\vert D_{g}^{\left( r\right) }\right\vert }%
H\left( g,\mathcal{N}_{r}^{v}\right) ,
\end{equation*}
for all $v\in \mathcal{V}(\mathcal{G})$. Note that $\mu _{k}^{(r)}\left(
v\right) $ is locally measurable within a radius $r$ around each node $v$,
namely, each agent is able to distributedly compute $\mu _{k}^{r}\left(
v\right) $ for all $v\in \mathcal{V}\left( \mathcal{G}\right) $. Thus, from (%
\ref{Distributed Spectral Moment}), we have that the $k$-th spectral moment
can be computed via a simple distributed averaging of $\mu _{k}^{r}\left(
v\right) $.

\begin{remark}
The maximum order of the spectral moment that can be computed via this
distributed approach depends on the radius $r$ of the neighborhood that each
agent can reach. In particular, in order to compute the $k$-th moment, we
should be able to detect all the graphs $g\in \mathbf{I}_{k}$. It is easy to
prove that the $k$-ring with diameter $\left\lfloor k/2\right\rfloor $ is
the graph in $\mathbf{I}_{k}$ with the maximum diameter. As we mentioned
before, in order for a particular subgraph $h$ to be locally countable, the
radius $r$ must be greater or equal to $\left\lceil diam\left( h\right)
/2\right\rceil $. Hence, for a particular $r$, we can distributedly compute
moments up to an order $k_{\max }=2r+1$.
\end{remark}

\section{Conclusions and Future Research}

In this paper, we have derived explicit relationships between spectral
properties of a network and the presence of certain subgraphs. In
particular, we are able to express the spectral moments as a linear
combination of the embedding frequencies of these subgraphs. Since the
spectral properties are known to have a direct influence on the dynamics of
the network, our result builds a bridge between local network measurements
(i.e., the presence of small subgraphs) and global dynamical behavior (via
the spectral moments). Furthermore, based on our result, we have also
developed a novel decentralized algorithm to efficiently aggregate a set of
local network measurements into global spectral moments. Our approach is
based on a an efficient decentralized approach to compute the embedding
frequencies of subgraphs from local samples of the network topology.

Future work involves to extend our methodology to directed graphs and graphs
with self-loops, such as those appearing in transcription networks. Also, we
are interested in developing techniques to extract explicit information
regarding the dynamical behavior of a network from a sequence of spectral
moments. Furthermore, it would be interesting to find a fully decentralized
algorithm to iteratively modify the structure of a network of agents in
order to control its dynamical behavior. We find particularly interesting
the case in which individual agents have knowledge of their local network
structure only (i.e., myopic information), while they try to collectively
aggregate these local pieces of information to find the most beneficial
modification of the network structure.

\bigskip




\begin{thebibliography}{99}
\bibitem{S01} S.H. Strogatz, \textquotedblleft Exploring Complex
Networks,\textquotedblright\ \emph{Nature}, vol. 410, pp. 268-276, 2001.

\bibitem{BLMCH06} S. Boccaletti S., V. Latora, Y. Moreno, M. Chavez, and
D.-H. Hwang, \textquotedblleft Complex Networks: Structure and
Dynamics,\textquotedblright\ \emph{Physics Reports}, vol. 424, no. 4-5, pp.
175-308, 2006.

\bibitem{Ald82} D.J. Aldous, \emph{Some Inequalities for Reversible Markov
Chains}, J. London Math. Soc. (2)25, pp. 564-576, 1982.

\bibitem{DGM08} M. Draief, A. Ganesh, and L. Massouli\'{e},
\textquotedblleft Thresholds for Virus Spread on
Networks,\textquotedblright\ \emph{Annals of Applied Probability}, vol. 18,
pp. 359-378, 2008.

\bibitem{PJ09} V.M. Preciado, and A. Jadbabaie, \textquotedblleft Spectral
Analysis of Virus Spreading in Random Geometric Networks,\textquotedblright\ 
\emph{Proc. of the 48th IEEE Conference on Decision and Control}, pp.
4802-4807, 2009.

\bibitem{PC98} L.M. Pecora, and T.L. Carroll, \textquotedblleft Master
Stability Functions for Synchronized Coupled Systems,\textquotedblright\ 
\emph{Physics Review Letters}, vol. 80(10), pp. 2109-2112, 1998.

\bibitem{PV05} V.M. Preciado, and G.C. Verghese, \textquotedblleft
Synchronization in Generalized Erd\"{o}s-R\'{e}nyi Networks of Nonlinear
Oscillators,\textquotedblright\ \emph{Proc. of the 44th IEEE Conference on
Decision and Control}, pp. 4628-4633, 2005.

\bibitem{P08} V.M. Preciado, \emph{Spectral Analysis for Stochastic Models
of Large-Scale Complex Dynamical Networks},\ Ph.D. dissertation, Dept.
Elect. Eng. Comput. Sci., MIT, Cambridge, MA, 2008.

\bibitem{PZJP10} V.M. Preciado, M.M. Zavlanos, A. Jadbabaie and G.J. Pappas,
\textquotedblleft Distributed Control of the Laplacian Spectral Moments of a
Network,\textquotedblright\ \emph{Proc. of the American Control Conference,
2010}. Accepted for publication.

\bibitem{MSIKCA02} R. Milo, S. Shen-Orr, S. Itzkovitz, N. Kashtan, D.
Chklovskii, and U. Alon, \textquotedblleft Network Motifs: Simple Building
Blocks of Complex Networks,\textquotedblright\ \emph{Science}, vol. 298, pp.
824 - 827, 2002.

\bibitem{Alon07} U. Alon, \textquotedblleft Network Motifs: Theory and
Experimental Approaches,\textquotedblright\ \emph{Nature Reviews Genetics,}
Vol. 8, pp. 450-461, 2007.

\bibitem{PB05} I. Popescu and D. Bertsimas, \textquotedblleft An SDP
Approach to Optimal Moment Bounds for Convex Classes of
Distributions,\textquotedblright\ \emph{Mathematics of Operation Research},
vol. 50, pp. 632-657, 2005.

\bibitem{KM08} D. Kempe, and F. McSherry, \textquotedblleft A Decentralized
Algorithm for Spectral Analysis,\textquotedblright\ \emph{Journal of
Computer and System Science}, vol. 74, pp. 70-83, 2008.

\bibitem{Big93} N. Biggs, \emph{Algebraic Graph Theory}, Cambridge
University Press, 2$^{nd}$ Edition, 1993.

\bibitem{BW04} G.R. Brightwell and P. Winkler, \textquotedblleft Note on
Counting Eulerian Circuits\textquotedblright , \emph{CDAM Research Report},
2004.

\bibitem{RSMOB02} E. Ravasz, A.L. Somera, D.A. Mongru, Z. Oltvai, and A.-L.
Barabasi, \textquotedblleft Hierarchical organization of modularity in
metabolic networks\textquotedblright , \emph{Science,} vol. 297, pp.
1551-1555, 2002.
\end{thebibliography}
\end{document}